\documentclass[12pt]{article}

\usepackage[T1]{fontenc}
\usepackage[utf8]{inputenc}

\usepackage{XCharter}

\usepackage{amsmath, amsthm}      
\usepackage[xcharter]{newtxmath}  
\usepackage{aliascnt}             

\usepackage{stmaryrd}  
\usepackage{dsfont}    

\usepackage[dvipsnames]{xcolor}  
\usepackage{graphicx}             
\usepackage{multicol}             
\usepackage{caption}              
\usepackage{subcaption}           

\usepackage[left=1.2in, right=1.2in, top=1.0in, bottom=0.8in, includehead, includefoot]{geometry}

\usepackage{setspace}             
\usepackage[title,toc,titletoc,page]{appendix}  

\usepackage{epigraph}
\usepackage{tcolorbox}  

\usepackage{imakeidx}      
\usepackage{authorindex}   

\usepackage[colorlinks=true, citecolor=Blue, linkcolor=blue]{hyperref}

\newcommand\myshade{85}
\colorlet{mylinkcolor}{violet}
\colorlet{mycitecolor}{PineGreen}
\colorlet{myurlcolor}{Aquamarine}
\hypersetup{
  linkcolor  = mylinkcolor!\myshade!black,
  citecolor  = mycitecolor!\myshade!black,
  urlcolor   = myurlcolor!\myshade!black,
  colorlinks = true,
}

\usepackage[capitalise,noabbrev]{cleveref}

\crefname{assumption}{Assumption}{Assumptions}
\Crefname{assumption}{Assumption}{Assumptions}
\crefname{algocf}{Algorithm}{Algorithms}
\Crefname{algocf}{Algorithm}{Algorithms}

\usepackage{natbib}  

\usepackage{enumitem}

\setlist[enumerate]{itemsep=2pt,topsep=3pt}
\setlist[itemize]{itemsep=2pt,topsep=3pt}

\setlist[enumerate,1]{label={\upshape (\roman*)}}

\usepackage[linesnumbered, ruled]{algorithm2e}
\SetKwBlock{Repeat}{repeat}{}  
\SetAlCapSkip{1.2em}           

\usepackage{tikz}
\usepackage{pifont}     
\usepackage{tikz-cd}    

\usetikzlibrary{positioning}
\usetikzlibrary{arrows}
\usetikzlibrary{calc}
\usetikzlibrary{intersections}
\usetikzlibrary{matrix}
\usetikzlibrary{decorations}
\usetikzlibrary{shapes, fit}
\usetikzlibrary{arrows.meta}
\usetikzlibrary{decorations.pathreplacing}
\usetikzlibrary{calligraphy}

\usepackage{pgf}
\usepackage{pgfplots}
\pgfplotsset{compat=1.16}

\newcommand{\navy}[1]{\textcolor{Blue}{\textit{#1}}}

\definecolor{codebg}{rgb}{0.95, 0.93, 0.93}
\definecolor{pale}{RGB}{235, 235, 235}
\definecolor{pale2}{RGB}{255,255,255}
\definecolor{aquamarine}{RGB}{69,139,116}

\newcommand{\EE}{\mathbb E}  
\newcommand{\RR}{\mathbb R}  
\newcommand{\NN}{\mathbb N}  
\newcommand{\PP}{\mathbb P}  
\newcommand{\TT}{\mathbb T}  

\newcommand{\Asf}{\mathsf A}

\newcommand{\Xsf}{\mathsf X}

\newcommand{\bB}{\mathscr B}

\newcommand{\fF}{\mathcal F}

\newcommand{\gmin}{g_{\triangledown}}

\renewcommand{\leq}{\leqslant}
\renewcommand{\geq}{\geqslant}

\renewcommand{\phi}{\varphi}
\renewcommand{\epsilon}{\varepsilon}

\newcommand{\setntn}[2]{ \{ #1 : #2 \} }          

\newcommand{\given}{\, | \,}                      
\newcommand{\cf}[1]{ \lstinline|#1| }             
\newcommand{\1}{\mathds{1}}                       
\newcommand*\diff{\mathop{}\!\mathrm{d}}          

\definecolor{thmfontcolor}{RGB}{0,0,0}

\newtheoremstyle{coloredplain}
  {\topsep}   
  {\topsep}   
  {\itshape}  
  {}          
  {\bfseries} 
  {.}         
  {.5em}      
  {}          

\theoremstyle{coloredplain}

\newtheorem{theorem}{Theorem}[section]
\newaliascnt{corollary}{theorem}

\aliascntresetthe{corollary}
\newaliascnt{lemma}{theorem}
\newtheorem{lemma}[lemma]{Lemma}
\aliascntresetthe{lemma}
\newaliascnt{proposition}{theorem}
\newtheorem{proposition}[proposition]{Proposition}
\aliascntresetthe{proposition}

\theoremstyle{definition}
\newtheorem{definition}{Definition}[section]
\newtheorem{example}{Example}[section]
\newtheorem{remark}{Remark}[section]
\newtheorem{assumption}{Assumption}[section]
\newaliascnt{condition}{assumption}

\aliascntresetthe{condition}

\crefname{lemma}{Lemma}{Lemmas}
\Crefname{lemma}{Lemma}{Lemmas}
\crefname{corollary}{Corollary}{Corollaries}
\Crefname{corollary}{Corollary}{Corollaries}
\crefname{proposition}{Proposition}{Propositions}
\Crefname{proposition}{Proposition}{Propositions}

\title{Dynamic Programming without Contractivity}

\author{
    Chengyuan Peng\thanks{Australian National University. Email: \texttt{chengyuanp2022@gmail.com}}
\and John Stachurski\thanks{National Graduate Institute for Policy Studies (GRIPS). Email: \texttt{j-stachurski@grips.ac.jp}}
\and Jingni Yang\thanks{School of Economics, University of Sydney, Australia. Email: \texttt{jingni.yang@sydney.edu.au}}
}

\date{\today}

\begin{document}

\maketitle

\begin{abstract}
    Optimality results in dynamic programming are usually derived from
    contraction arguments, which require discounting or uniform bounds on
    termination times.  We show that, in many useful settings, contractivity can be replaced by a simple global
    stability condition that requires no rate or uniformity.   Under
    this condition and mild order-theoretic assumptions, we obtain existence
    of optimal policies, uniqueness of the value function as a solution to
    the Bellman equation, and convergence of value function iteration and of
    Howard and optimistic policy iteration.  Applications include optimal
    stopping without discounting, where contraction methods can fail, and
    Bayesian sequential analysis.
    \\

\end{abstract}

\medskip

\section{Introduction}\label{s:intro}

The standard route to optimality results in dynamic programming runs through
contraction mappings.  When rewards are bounded and future payoffs are
discounted, the policy operators and the Bellman operator are contractions
with respect to the supremum norm, and the Banach contraction mapping theorem
delivers existence and uniqueness of the value function, the Bellman principle
of optimality, and geometric convergence of value function iteration 
\citep{denardo1967contraction,puterman2005markov,bertsekas2022abstract}.  The same approach extends to weighted
supremum norms and to operators that are only eventually contracting, in the
sense that some iterate is a contraction (see, e.g.,
\cite{stachurski2021dynamic} or \cite{toda2021perov}).

Many important problems involve no discounting, however, or even permanently
negative discount rates. In such problems the standard contraction-based methods
fail. Leading examples include no-discount optimal stopping, positive control
systems, sequential
analysis, optimized production chains, economic geography, and stochastic
shortest path problems.  For background, see, for example,
\cite{wald1947sequential, arrow1949bayes, bertsekas1991stochastic,
lai2001sequential,
hsu2014optimal,
kikuchi2021coase,
bertsekas2022abstract,
rantzer2022explicit,
tyazhelnikov2022production,
li2024exact,
ohlin2024optimal,
pates2024optimal,
ohlin2026heuristic}.

To handle such applications we replace
contractivity with global stability, a significantly weaker condition.   We show that, for regular and globally
stable dynamic programs, existence of a fixed point of the Bellman operator
implies the fundamental optimality properties---existence of an optimal policy,
uniqueness of the value function as a solution to the Bellman equation within
the relevant class, and Bellman's principle of optimality---together with
convergence of the three core algorithms: value function iteration, optimistic
policy iteration and Howard policy iteration.  We then show that the same
conclusions hold when the existence of a fixed point is replaced by two
conditions on primitives: order boundedness of the policy operators and a weak
order completeness restriction on the value space.

We work in the generalized dynamic programming environment introduced in
\cite{sargent2025partially}, with one addition: the value spaces (sets of
candidate value functions, or cost-to-go functions) are
\emph{pospaces}---topological spaces endowed with a closed partial
order---rather than just partially ordered sets (posets). We leverage this
additional structure to obtain the more applicable results outlined above.
Moreover, for global stability, we require only that each policy operator has a
unique fixed point to which all of its trajectories converge in the pospace
topology.  The topology need not be generated by a norm.  In fact, in our main
application it is the topology of pointwise convergence.   The Bellman operator
itself need not be globally stable, so restrictions are only placed on the
(simpler) policy operators.

Our results are related to those in \cite{bertsekas2017regular}, which also
condition on global stability, and were likewise developed to handle
noncontractive dynamic programs.  The major difference is that, as mentioned
above, we work in the generalized environment of \cite{sargent2025partially},
which was designed to accommodate the rich set of dynamic programming problems
found in modern applications---many of which fall outside the setting of
\cite{bertsekas2017regular} and the related monograph
\cite{bertsekas2022abstract}. These include dynamic programs built directly on
Q-factors, structural estimation problems, distributional dynamic programs,
approximate dynamic programs built around Monte Carlo, and combinations of the
above \citep{watkins1989learning, kochenderfer2022algorithms,
rust1994structural, mogensen2018solving, bellemare2017distributional,
rust2002there, munos2008finite, kalathil2021empirical, sidford2023variance}.

We apply these results to two problems.  The first is optimal stopping
without discounting on a general measurable state space, a problem with a
long history beginning with \cite{wald1947sequential} and
\cite{chow1963optimal} (see also \cite{shiryaev2007optimal} and
\cite{peskir2006optimal}).  We assume only that the policy which stops when
the exit cost does not exceed the flow cost terminates with probability one
and generates uniformly bounded expected cumulative flow costs.  Under this
assumption we obtain the full set of optimality properties together with
convergence of the three main algorithms (value function iteration, optimistic
policy iteration, and Howard policy iteration).  

The second application is a classical Bayesian sequential analysis problem along
the lines of \cite{wald1947sequential} and \cite{arrow1949bayes} (see Chapter~5
of \cite{bertsekas1976dynamic} for a textbook treatment), which we treat as a
special case of the optimal stopping problem.  Using a bound on exit times for
bounded martingales, we show that the fundamental optimality properties hold and
that all three algorithms converge under the minimal condition that the two
candidate densities differ on a set of positive measure.  

The remainder of the paper is structured as follows.  Section~\ref{s:prel}
reviews pospaces and abstract dynamic programs.  Section~\ref{s:or} states and
proves the optimality results.  Section~\ref{s:a} treats the two
applications, and Section~\ref{s:conclusion} discusses open problems.  The appendix contains a bound on exit times for bounded
martingales.

\section{Preliminaries}\label{s:prel}

This section collects preliminaries from order theory and topology, and then
reviews abstract dynamic programs, their optimality properties, and the core
algorithms used to solve them.

\subsection{Mathematical Preliminaries}\label{ss:mp}

We begin with the definitions and results from order theory and topology that
are used in the rest of the paper.

Let $V$ be a \navy{partially ordered space}, or \navy{pospace}, meaning that
$V = (V, \preceq)$ is a partially ordered set endowed with a topology under
which $\preceq$ is \navy{closed}: for every directed index set $\Lambda$ and
any nets $(u_\alpha)_{\alpha \in \Lambda}$ and $(v_\alpha)_{\alpha \in
\Lambda}$ in $V$,
\begin{equation*}
    u_\alpha \to u,
    \;\;
    v_\alpha \to v
    \;
    \text{ and }
    \;
    u_\alpha \preceq v_\alpha \text{ for all } \alpha \in \Lambda
    \quad \implies \;
    u \preceq v.
\end{equation*}
An \navy{order interval} in $V$ is a set
of the form $\setntn{v \in V}{a \preceq v \preceq b}$ for some $a,b \in V$.
When $V$ is understood, we denote this set by $[a,b]$. A set $A \subset V$ is called
\navy{order bounded} if there exists an order interval $[a,b] \subset V$ with $A
\subset [a,b]$.
A sequence $(v_n)$ in $V$ is called \navy{decreasing} if $v_{n+1} \preceq v_n$
for all $n \in \NN$.  If $(v_n)$ is decreasing and the infimum $\bigwedge_n v_n$
is equal to some $v \in V$, then we write $v_n \downarrow v$.
The set $V$ is called \navy{monotone sigma-complete} if every decreasing
sequence in $V$ that is bounded below has an infimum in $V$.
A self-map $S \colon V \to V$ is called \navy{order preserving}
if $u \preceq v$ implies $Su \preceq Sv$, and \navy{globally stable} if $S$
has a unique fixed point $\bar v \in V$ and $S^k u \to \bar v$ as $k \to
\infty$ for all $u \in V$.

We use the following elementary result. In the statement, 
$(v_\alpha)_{\alpha \in \Lambda}$ is a net in $V$.

\begin{lemma}\label{l:posmcoc}
    If $v_\alpha \to v$ and $v \preceq v_\alpha$ for all $\alpha \in \Lambda$,
    then $\bigwedge_\alpha v_\alpha = v$.
\end{lemma}

\begin{proof}
    By assumption, $v$ is a lower bound of $(v_\alpha)$.  If $w$ is any other
    lower bound, then $w \preceq v_\alpha$ for all $\alpha$.  Since $\preceq$
    is closed and $v_\alpha \to v$, this implies $w \preceq v$.  Hence $v$ is
    the greatest lower bound of $(v_\alpha)$.
\end{proof}

The next lemma is the main bridge between topology and order in what
follows.\footnote{In the terminology of \cite{sargent2025partially},
\cref{l:comp} states that every order preserving and globally stable
self-map on a pospace is \emph{order stable}.}

\begin{lemma}\label{l:comp}
    Let $S$ be an order preserving self-map on $V$ that is globally stable
    with fixed point $\bar v$.  For all $v \in V$,
    \begin{enumerate}
        \item $v \preceq S v$ implies $v \preceq \bar v$, and
        \item $S v \preceq v$ implies $\bar v \preceq v$.
    \end{enumerate}
\end{lemma}

\begin{proof}
    If $v \in V$ and $v \preceq S \, v$, then, iterating
    on this inequality and using the fact that $S$ is order preserving, we have
    $v \preceq S^n v$ for all $n \in \NN$. Since the partial order is closed and
    $S^n v \to \bar v$, taking the limit gives $v \preceq \bar v$.
    The proof of (ii) is similar.
\end{proof}

\subsection{Abstract Dynamic Programs}\label{ss:ba}

Abstract dynamic programs in the sense of the term used here were introduced in
\cite{sargent2025partially}.  To make the paper self-contained, we
review the key definitions.  Since our applications are cost minimization
problems, we state all definitions and results for the case of minimization.
Maximization problems are handled by reversing the partial order, an operation
that preserves all of the assumptions imposed below.

Slightly specializing \cite{sargent2025partially}, who work in the setting of
partially ordered sets without any topology, we define an \navy{abstract
dynamic program} (ADP) to be a pair $(V, \TT)$, where
\begin{enumerate}
    \item $V = (V, \preceq)$ is a pospace and
    \item $\TT = \setntn{T_\sigma}{\sigma \in \Sigma}$ is a nonempty family of
        order preserving self-maps on $V$.
\end{enumerate}
In what follows,
\begin{itemize}
    \item $V$ is called the \navy{value space}.
    \item Each operator $T_\sigma$ in $\TT$ is called a \navy{policy
        operator}.
    \item $\Sigma$ is an arbitrary index set and elements of $\Sigma$
        will be referred to as \navy{policies}.
\end{itemize}

In applications we will impose conditions under which each $T_\sigma$
has a unique fixed point in $V$.  In these settings, the significance of $T_\sigma$ is
that its fixed point represents the lifetime cost of following policy
$\sigma$.  We will denote the unique fixed point of $T_\sigma$
by $v_\sigma$ and call it the \navy{$\sigma$-value function}.

Let $(V, \TT)$ be an ADP with policy set $\Sigma$. Given $v \in V$, we
\begin{itemize}
    \item say that $\sigma \in \Sigma$  is \navy{$v$-greedy} if $T_\sigma \, v
        \preceq T_\tau \, v$ for all $\tau \in \Sigma$, and
    \item set $V_G \coloneq \setntn{v \in V}{\text{ at least one $v$-greedy
        policy exists}}$.
\end{itemize}
In many applications, solving for a $v$-greedy policy is straightforward.
Moreover, there exist conditions under which solving the overall problem reduces
to solving for a $v$-greedy policy with a ``correct'' choice of $v$. This idea
is at the heart of dynamic programming and the details are clarified below.

We now list properties  useful for ADP optimality theory.
We call $(V, \TT)$
\begin{itemize}
    \item \navy{well-posed} if each $T_\sigma \in \TT$ has a unique fixed point in $V$,
    \item \navy{globally stable} if each $T_\sigma \in \TT$ is globally stable on $V$,
    \item \navy{regular} if $V_G = V$; that is, if a $v$-greedy policy exists
        for all $v \in V$, and
    \item \navy{order bounded} if there exists a $b \in V$ with $b \preceq
        T_\sigma \, b$ for all  $T_\sigma \in \TT$.
\end{itemize}
Well-posedness is a minimal precondition for a coherent  dynamic
program: in order to optimize over policies, we need well-defined lifetime values.
When well-posedness holds we will
always use $v_\sigma$ to denote the unique fixed point of $T_\sigma$.
Also, we set
\begin{equation*}
    V_\Sigma \coloneq \text{ all $v \in V$ such that $T_\sigma \, v = v$ for
    some $T_\sigma \in \TT$}.
\end{equation*}
Obviously, if $(V, \TT)$ is globally stable, then $(V, \TT)$ is well-posed.
Global stability is the central assumption of this paper.  Regularity  helps
us to construct algorithms and to obtain existence of optimal policies.  Order
boundedness is a useful property for optimality and is implied by, but does
not imply, the value space $V$ itself being order bounded.

We define the \navy{Bellman operator} generated by $(V, \TT)$  via
\begin{equation*}
    Tv \coloneq \bigwedge_{\sigma} T_\sigma \, v
    \quad \text{whenever the infimum exists}.
\end{equation*}
In addition, we say that $v \in V$ satisfies the \navy{Bellman equation} if
\begin{equation*}
    v = \bigwedge_{\sigma} T_\sigma \, v.
\end{equation*}
Here the infimum is taken over all $\sigma \in \Sigma$.
There is, in general, no guarantee that the infimum exists.
Evidently $v \in V$ satisfies the
Bellman equation if and only if $Tv$ exists and $Tv = v$.
Also, $\sigma$ is $v$-greedy if and only if $T_\sigma \, v = T v$.
If $(V, \TT)$ is regular, then $T$ is defined on all of $V$, and $T$ is order
preserving, since $u \preceq w$ implies $T u \preceq T_\sigma \, u \preceq
T_\sigma \, w$ for all $\sigma$, and hence $Tu \preceq Tw$.

\subsection{Defining Optimality}\label{ss:opt}

Suppose that $(V, \TT)$ is a well-posed ADP.  In this setting,
we say that a policy $\sigma \in \Sigma$ is \navy{optimal} for
$(V, \TT)$ if $v_\sigma$ is a least element of $V_\Sigma$.
In other words, $\sigma$ is optimal if it attains the ``lowest possible'' lifetime cost.

Perhaps the most important aspect of the theory of dynamic
programming is the link between optimality and the Bellman equation.
To clarify this link we introduce some terminology and set
\begin{equation*}
    v^* \coloneq \bigwedge_\sigma v_\sigma = \bigwedge V_\Sigma
    \quad \text{whenever the infimum exists}.
\end{equation*}
When $v^*$ exists (i.e., when the infimum exists) we call $v^*$ the
\navy{value function} of the ADP.
The following statements are obvious from the definitions:
\begin{itemize}
    \item Existence of an optimal policy $\sigma$ implies that $v^*$ exists and is equal to $v_\sigma$.
    \item If $v^* = v_\sigma$ for some $\sigma \in \Sigma$, then $\sigma$ is
        optimal.
\end{itemize}
At the same time, existence of $v^*$ does not generally imply existence of a
least element (and hence an optimal policy).

\begin{definition}
    We say that \navy{Bellman's principle of optimality holds} if
    \begin{equation}\label{eq:bpo}
      \setntn{\sigma \in \Sigma}{\text{$\sigma$ is optimal }}
      =
      \setntn{\sigma \in \Sigma}{\text{$\sigma$ is $v^*$-greedy }}.
    \end{equation}
\end{definition}

(When $v^*$ does not exist both sets are understood as empty.)

\begin{definition}\label{d:fops}
    We say that the \navy{fundamental optimality properties hold} if
    \begin{enumerate}\label{enum:b13}
        \item[(B1)] at least one optimal policy exists,
        \item[(B2)] $v^*$ exists and is the unique solution to the Bellman equation in $V_G$, and
        \item[(B3)] Bellman's principle of optimality holds.
    \end{enumerate}
\end{definition}

These definitions are essentially the same as those in
\cite{sargent2025partially}.  The only differences are that (a) we treat
minimization instead of maximization, and (b) $V$ is replaced by
$V_G$ in (B2).  This leads to cleaner results in settings where not every $v \in
V$ has a greedy policy.

In terms of applications, all three parts of \cref{d:fops} are significant.
(B1) tells us that the problem at hand has a solution.  (B3) implies that a
solution can be computed by taking a $v^*$-greedy policy, and that any other
optimal policy must also be $v^*$-greedy. Finally, (B2) provides a restriction
that can help us calculate $v^*$. Provided that we search in $V_G$, any fixed
point of $T$ is equal to $v^*$.

\subsection{Core Algorithms}\label{ss:alg}

Next we review the three core algorithms for solving dynamic
programs: value function iteration (VFI), Howard policy iteration (HPI), and
optimistic policy iteration (OPI).  For background on the role and significance
of these algorithms, see \cite{bertsekas2022abstract},
\cite{sargent2025dynamic}, or \cite{sargent2025partially}.

Let $(V, \TT)$ be a regular and well-posed ADP. As in \cite{sargent2025partially}, we
define the \navy{Howard policy operator} corresponding to $(V, \TT)$ via
\begin{equation*}
    H \colon V \to V_\Sigma, \qquad
    H v = v_\sigma \quad \text{ where $\sigma$ is $v$-greedy}.
\end{equation*}
For each $m \in \NN$, we define the \navy{optimistic policy operator}
\begin{equation}\label{eq:wmop}
    W \colon V \to V, \qquad
    W v
    = T^m_\sigma v \quad \text{where } \sigma \text{ is $v$-greedy}.
\end{equation}
Here and below, the dependence of $W$ on $m$ is often suppressed to simplify
notation.  Note that, for both of these maps, we always select the same
$v$-greedy policy when applying them to some fixed $v$.\footnote{In general,
designating a $v$-greedy policy for all $v \in V$ requires the Axiom of
Choice. In practice, many applications induce some structure on the policy set
that can be used to produce simple selection mechanisms.}

As in \cite{sargent2025partially}, we associate VFI, OPI, and HPI with fixing a
$v \in V$ and then iteratively applying the operators $T$, $W$, and $H$
respectively.
Suppose that $(V, \TT)$ is a regular ADP with Bellman operator $T$ and
the fundamental optimality
properties hold. Let $v^*$ denote the value function. We set
\begin{itemize}
    \item $V_D \coloneq$ all $v \in V$ with $Tv \preceq v$.
\end{itemize}
In this setting, we say that
\navy{VFI converges} if $T^n v \downarrow v^*$ for all $v \in V_D$.
We say that \navy{OPI converges} if the same holds with $T$ replaced by $W$.
We say that \navy{HPI converges} if the same holds with $T$ replaced by $H$.
For OPI convergence, the meaning is
that convergence holds for any choice of the OPI step size $m \in \NN$.

\begin{proposition}\label{l:cl2}
   Let $(V, \TT)$ be regular and globally stable, and suppose that the
   fundamental optimality properties hold.  If $v \in V_D$, then the
   sequences $(T^n v)$, $(W^n v)$ and $(H^n v)$ are decreasing and, for all
   $n \in \NN$,
   \begin{equation}\label{eq:rank}
       v^* \preceq W^n v \preceq T^n v
       \quad \text{and} \quad
       v^* \preceq H^n v \preceq T^n v.
   \end{equation}
   In particular, convergence of VFI implies convergence of OPI and HPI.
\end{proposition}

\begin{proof}
    Fix $v \in V_D$ and let $\sigma$ be the $v$-greedy policy used to
    define $W v$ and $H v$.  We claim that
    \begin{equation}\label{eq:vhv}
        v^*
        \preceq H v
        \preceq W v
        \preceq T v
        \preceq v.
    \end{equation}
    The last inequality holds because $v \in V_D$.  Since $\sigma$ is
    $v$-greedy, we have $T_\sigma \, v = T v \preceq v$.  Iterating on this
    inequality shows that $(T^k_\sigma \, v)_{k \geq 0}$ is decreasing, so $W v
    = T^m_\sigma \, v \preceq T_\sigma \, v = T v$.  Moreover, $T_\sigma (W v)
    = T_\sigma^{m+1} v \preceq T_\sigma^m v = W v$, so \cref{l:comp} gives
    $H v = v_\sigma \preceq W v$.  Finally, $v^* \preceq v_\sigma = Hv$ by the
    definition of $v^*$.  This proves \eqref{eq:vhv}.

    Next we show that $T$, $W$ and $H$ map $V_D$ to itself.  For $T$, the
    inequality $T v \preceq v$ and order preservation give $T (T v) \preceq
    T v$.  For $W$, we have $T (W v) \preceq T_\sigma (W v) \preceq W v$,
    where the second inequality was shown in the previous paragraph.  For
    $H$, we have $T (H v) \preceq T_\sigma \, v_\sigma = v_\sigma = H v$.

    Combining these facts with \eqref{eq:vhv} shows that the three sequences
    are decreasing and bounded below by $v^*$.  To prove the remaining
    inequalities in \eqref{eq:rank}, observe that, if $u, w \in V_D$ with $w
    \preceq u$, then \eqref{eq:vhv} applied to $w$ and order preservation of
    $T$ give $W w \preceq T w \preceq T u$ and $H w \preceq T w \preceq T
    u$.  Induction on $n$, using the invariance of $V_D$, then yields $W^n v
    \preceq T^n v$ and $H^n v \preceq T^n v$ for all $n$.

    Regarding the final claim, suppose that $T^n v \downarrow v^*$.  If $\ell$
    is a lower bound of $(W^n v)$, then, by \eqref{eq:rank}, $\ell$ is also a
    lower bound of $(T^n v)$, so $\ell \preceq v^*$.  Since $v^*$ is itself a
    lower bound of $(W^n v)$, we have $\bigwedge_n W^n v = v^*$.  As $(W^n
    v)$ is decreasing, this gives $W^n v \downarrow v^*$.  The proof for HPI
    is identical.
\end{proof}

\section{Optimality Results}\label{s:or}

Next we state and prove our main results. 

\subsection{A Necessary and Sufficient Condition}

We begin with a necessary and sufficient condition for optimality, which, in the setting of
ADPs, is relatively elementary.

\begin{proposition}\label{p:foo}
    The fundamental optimality properties hold if and only if $v^*$ exists and
    is the unique solution to the Bellman equation in $V_G$.
\end{proposition}

\begin{proof}
    The claim ($\Rightarrow$) is trivial.  Regarding ($\Leftarrow$), suppose
    that $v^*$ exists and is the unique solution to the Bellman equation in
    $V_G$.  (B2) holds by the hypothesis.  Since $v^* \in V_G$, we can take
    $\sigma \in \Sigma$ to be $v^*$-greedy, in which case $T_\sigma \, v^* = T
    v^* = v^*$.  By well-posedness, $v_\sigma = v^*$, so $\sigma$ is optimal.
    This gives (B1).  Regarding (B3), if $\sigma$ is optimal, then $v_\sigma =
    v^*$, so $T_\sigma \, v^* = T_\sigma \, v_\sigma = v_\sigma = v^* = T v^*$,
    and hence $\sigma$ is $v^*$-greedy.  Conversely, if $\sigma$ is
    $v^*$-greedy, then $T_\sigma \, v^* = T v^* = v^*$, so $v_\sigma = v^*$ by
    well-posedness, and $\sigma$ is optimal.  This gives (B3).
\end{proof}

\cref{p:foo} tells us our main goal is to construct conditions under which
$v^*$ exists and is the unique fixed point of $T$ in $V_G$.  We begin this
task in Section~\ref{ss:scs}.

\subsection{Sufficient Conditions for Optimality and Convergence}\label{ss:scs}

Next we impose regularity and seek conditions under which 
global stability provides strong optimality and convergence properties.
Regularity is added so that the three core algorithms are well-defined.

\begin{theorem}\label{t:pospace}
    Let $(V, \TT)$ be regular and globally stable.  If $T$ has a fixed point
    in $V$, then
    \begin{enumerate}
        \item the fundamental optimality properties hold and
        \item VFI, OPI and HPI all converge.
    \end{enumerate}
\end{theorem}

\begin{proof}
    Let $\bar v$ be a fixed point of $T$ in $V$.  By regularity, we can take
    a $\bar v$-greedy policy $\sigma$, in which case $T_\sigma \, \bar v = T
    \bar v = \bar v$, and hence $\bar v = v_\sigma \in V_\Sigma$.  Moreover,
    if $\tau \in \Sigma$, then $\bar v = T \bar v \preceq T_\tau \, \bar v$, so
    $\bar v \preceq v_\tau$ by \cref{l:comp}.  Hence $\bar v$ is the least
    element of $V_\Sigma$.  This proves that $v^*$ exists and equals $\bar
    v$.  Since $\bar v$ was an arbitrary fixed point of $T$ and $V_G = V$, the
    same argument shows that $v^*$ is the unique solution to the Bellman
    equation in $V_G$.  Part (i) now follows from \cref{p:foo}.

    Regarding convergence of VFI, fix $v \in V_D$ and let $\sigma$ be an
    optimal policy.  By \cref{l:cl2}, the sequence $(T^n v)$ is decreasing
    and $v^* \preceq T^n v$ for all $n$.  Also, since $T \preceq T_\sigma$
    pointwise on $V$ and both operators are order preserving, induction gives
    $T^n v \preceq T_\sigma^n \, v$ for all $n$.  As $\sigma$ is optimal and
    $T_\sigma$ is globally stable, we have $T_\sigma^n \, v \to v_\sigma =
    v^*$.  Since $v^* \preceq T_\sigma^n \, v$ for all $n$, \cref{l:posmcoc}
    implies that $\bigwedge_n T^n_\sigma \, v = v^*$.  Now, if $\ell$ is any
    lower bound of $(T^n v)$, then $\ell$ is also a lower bound of
    $(T_\sigma^n \, v)$, so $\ell \preceq v^*$.  Hence $\bigwedge_n T^n v =
    v^*$, and, since $(T^n v)$ is decreasing, $T^n v \downarrow v^*$.  This
    proves that VFI converges.  Convergence of OPI and HPI now follows from
    \cref{l:cl2}.
\end{proof}

\cref{t:pospace} is a high-level result because a
condition (existence of a fixed point) is placed on the derived object $T$,
rather than the primitives $V$ and $\TT$.
The next theorem leverages \cref{t:pospace} while placing all assumptions
on primitives, using monotone sigma-completeness and order boundedness,
combined with global stability.

\begin{theorem}\label{t:tspo}
    Let $(V, \TT)$ be regular and globally stable.  If, in addition,
    $(V, \TT)$ is order bounded and $V$ is monotone sigma-complete, then
    \begin{enumerate}
        \item the fundamental optimality properties hold and
        \item VFI, OPI and HPI all converge.
    \end{enumerate}
\end{theorem}

\begin{proof}
    In view of \cref{t:pospace}, we need only show that
    $T$ has a fixed point in $V$.  To see this, we use the order bounded
    property to take a $b \in V$ such that $b \preceq T_\sigma \, b$ for all $\sigma \in
    \Sigma$. Now let $v$ be any element of $V_\Sigma$, and let $\tau \in
    \Sigma$ be such that $v = v_\tau$.  Since $b \preceq T_\tau \, b$,
    \cref{l:comp} gives $b \preceq v$.  Moreover, $T v \preceq T_\tau \, v =
    v$.  Since $T$ is order preserving, it follows that $T$ is a self-map on
    the order interval $[b, v]$.
    Letting $v_n \coloneq T^n v$ and applying monotone sigma-completeness, we
    have $v_n \downarrow \bar v$ for some $\bar v \in [b, v]$.  We claim that
    $T \bar v = \bar v$. To see this, first observe that $T \bar v \preceq T
    v_n = v_{n+1}$ for all $n$, so $T \bar v \preceq \bar v$. Letting
    $\sigma$ be $\bar v$-greedy, we have $T_\sigma \, \bar v = T \bar v
    \preceq \bar v$, so, by \cref{l:comp},
    $v_\sigma \preceq \bar v$.  Also, letting $w_n \coloneq T_\sigma^n \, v$,
    induction on $n$ gives $v_n \preceq w_n$ for all $n$, since
    $v_n = T \, v_{n-1} \preceq T_\sigma \, v_{n-1} \preceq T_\sigma \, w_{n-1} = w_n$
    by $T \preceq T_\sigma$ and order preservation of $T_\sigma$.  Combined with
    $v_\sigma \preceq \bar v \preceq v_n$, we have
    $v_\sigma \preceq \bar v \preceq v_n \preceq w_n$ for all $n \in \NN$.
    By global stability, $w_n \to v_\sigma$.  Applying \cref{l:posmcoc}, we
    also have $\bigwedge_n w_n = v_\sigma$.  From this fact and the previous chain of
    inequalities, it must be that $\bar v = v_\sigma$.  Hence $T \bar v =
    T_\sigma \, \bar v = T_\sigma \, v_\sigma = v_\sigma = \bar v$. This
    completes the proof that $\bar v$ is a fixed point of $T$.
\end{proof}

\begin{remark}\label{r:r}
    \cite{bertsekas2017regular} provides a related approach based on the stability of policy operators.
    He assumes that policy operators converge
    pointwise to their value function from every starting point in a set $S$ of
    functions on the state space.  Under this condition,
    \cite{bertsekas2017regular} obtains uniqueness of the solution to the Bellman
    equation and convergence of VFI when the optimal cost function over regular
    policies lies in $S$, and convergence of Howard and optimistic policy iteration
    when, in addition, the one-step operator is continuous along decreasing
    sequences. \Cref{t:tspo} requires neither of the last two conditions: existence of $v^*$ follows
    from order boundedness and monotone sigma-completeness, and no continuity of
    the Bellman operator is used. This difference matters for the optimal
    stopping problem we consider below, where the space is not countable, and
    $v^*$ is a pointwise infimum over an uncountable family of policies.
    In addition, our result is obtained in the generalized abstract
    dynamic programming environment of \cite{sargent2025partially}, which
    extends its reach to a very broad class of dynamic programs (as discussed in
    the introduction).
\end{remark}

\begin{remark}\label{r:ss2025}
    \cref{t:tspo} is also related to Theorem~5.5 of
    \cite{sargent2025partially}, which establishes analogous results in a
    similar setting, but using order stability and order continuity of the
    policy operators instead of our global stability condition.
    However, order stability is typically checked by verifying
    global stability (as in \cref{l:comp}).  Our results show that,
    once global stability has been checked, order continuity is unnecessary.
    This is significant because order continuity can
    fail even in standard settings.\footnote{For example, let $bc\Xsf$ be the
        set of bounded continuous functions on a metric space $\Xsf$.  A linear
        operator defined by conditional expectations, such as $f \mapsto \int
        f(y) P(\cdot, \diff y)$, typically fails to be order continuous when
        viewed as a map from $bc\Xsf$ to $bc\Xsf$, since infima of decreasing
    sequences in $bc\Xsf$ need not be pointwise limits.} 
\end{remark}

In the next section we apply these results to two problems.

\section{Applications}\label{s:a}

In this section we apply the preceding results to two applications:
no-discount optimal stopping and Bayesian sequential analysis.

\subsection{No-Discount Optimal Stopping}\label{ss:nodos}

Many important applications---including sampling problems, shortest path and
routing problems, bandit problems, and reinforcement learning tasks---involve no
discounting.  In the absence of a discount factor, the policy operators are
not contractions with respect to the supremum norm and, as we show below,
they need not be eventually contracting with respect to any equivalent norm.
In this section we study one such model: optimal stopping without
discounting.  (For background
and connections to the related literature, see the introduction.)
Later, in Section~\ref{ss:sar}, we apply the optimal stopping results
obtained below to a classical problem in sequential analysis.

\subsubsection{Model and policies}

Let $(\Xsf, \bB)$ be a measurable space and let $P$ be a
stochastic kernel on $\Xsf$. Let $(X_t)_{t\geq 0}$ be a
\navy{$P$-Markov} process, meaning that $\PP\{X_{t+1} \in B \given \fF_t\} = P(X_t, B)$
almost surely for all $t \geq 0$ and $B \in \bB$, where $(\fF_t)$ is the
natural filtration, and write $\PP_x$ and $\EE_x$ for probabilities
and expectations conditional on $X_0=x$. Let $b\Xsf$
denote the set of bounded $\bB$-measurable real-valued
functions on $\Xsf$, and let $b\Xsf_+$ denote its nonnegative
cone.

We consider a cost minimization problem with Bellman equation
\begin{equation}\label{eq:bell}
    g(x) = \min
    \left\{
        e(x), c(x) + \int g(x') P(x, \diff x')
    \right\},
\end{equation}
where $e, c \in b\Xsf_+$. The function $e$ is the exit cost
and $c$ is the flow cost. At each date, the controller either
stops and incurs $e(X_t)$, or continues, incurs $c(X_t)$, and
moves to the next state according to $P$.

A policy is a $\bB$-measurable map $\sigma$ from $\Xsf$ to $\{0, 1\}$,
with $\sigma(x) = 1$ indicating the decision to stop in state $x$.  
Its exit region is
\begin{equation*}
    E_\sigma \coloneq \setntn{x \in \Xsf}{\sigma(x) = 1},
\end{equation*}
and its associated stopping time is
\begin{equation}\label{eq:tau}
    \tau^\sigma 
    \coloneq \inf\setntn{t \geq 0}{X_t \in E_\sigma}
    = \inf\setntn{t \geq 0}{\sigma(X_t) = 1}.
\end{equation}
We use the convention $\inf \varnothing \coloneq \infty$. 
The expected cost generated by $\sigma$ is
\begin{equation}\label{eq:lcf}
    g_\sigma(x) \coloneq
        \EE_x
        \left[
            \sum_{t=0}^{\tau^\sigma-1} c(X_t) + e(X_{\tau^\sigma})
        \right].
\end{equation}
Here $\sum_{t=0}^{-1} c(X_t)$ is understood as $0$, so that $g_\sigma(x) = e(x)$
when $x \in E_\sigma$. The expression \eqref{eq:lcf} is well-defined
whenever $\tau^\sigma < \infty$ with $\PP_x$-probability one, which will be
the case under \cref{a:cap} below.  In that case $g_\sigma$ takes values in
$[0, \infty]$ and $g_\sigma(x)$ represents the
total expected cost when applying $\sigma$ in every period, conditional on
starting in state $x$.  Moreover, $g_\sigma$ is $\bB$-measurable, since $x
\mapsto \EE_x H$ is measurable for every nonnegative measurable function $H$
of the trajectory (see, e.g., Chapter~3 of \cite{meyn2009markov}).

A policy of particular interest is the \navy{lower bound policy} $\bar \sigma = \1\{e \leq c\}$.  Its exit
region, which we call the \navy{certain exit region}, is denoted
\begin{equation}\label{eq:cerex}
    \bar E \coloneq E_{\bar \sigma} = \setntn{x \in \Xsf}{e(x) \leq c(x)}
\end{equation}
and the stopping time is denoted
\begin{equation}\label{eq:bartau}
    \bar \tau 
    \coloneq \tau^{\bar \sigma}
    = \inf\setntn{t \geq 0}{\bar \sigma(X_t) = 1}.
\end{equation}
Under the convention that the controller stops whenever
indifferent, any optimal policy stops on $\bar E$. The reason is that
the controller has the opportunity to exit at cost $e(x) \leq c(x)$, and
continuing incurs $c(x)$ plus additional costs in subsequent stages.

Accordingly, we restrict attention to the policy space
\begin{equation}\label{eq:ndsig}
    \Sigma \coloneq 
    \{
        \text{all $\bB$-measurable $\sigma \colon \Xsf \to \{0,1\}$
        with $\bar \sigma \leq \sigma$}
    \}.
\end{equation}
For every $\sigma \in \Sigma$,
\begin{equation}\label{eq:sbou}
    \bar \sigma \leq \sigma \; \implies \; \tau^\sigma \leq \bar \tau \;\;
    \PP\text{-a.s.}
\end{equation}
Thus, $\bar \tau$, which we call the \navy{upper bound stopping time}, is an
upper bound on the stopping times generated by admissible policies.

Our key assumption is as follows.

\begin{assumption}\label{a:cap}
    The upper bound stopping time $\bar \tau$ is finite $\PP_x$-almost surely
    for every $x \in \Xsf$, and
    \begin{equation}\label{eq:capc}
        C \coloneq
        \sup_{x \in \Xsf} \EE_x \sum_{t=0}^{\bar \tau - 1} c(X_t) < \infty.
    \end{equation}
\end{assumption}

In words, the lower bound policy terminates with probability one and its
expected cumulative flow cost is uniformly bounded.  Since $\bar \tau = 0$
with probability one when $x \in \bar E$, it suffices to check these
conditions for $x \in \bar E^c$.  Because $c$ is bounded, a sufficient
condition for \cref{a:cap} is
\begin{equation}\label{eq:unifexit}
    \sup_{x \in \Xsf} \EE_x \bar \tau < \infty,
\end{equation}
in which case $C \leq \|c\| \sup_x \EE_x \bar \tau$.  (Here and below,
$\| \cdot \|$ is the supremum norm.)

\begin{lemma}\label{l:fomc}
     If \cref{a:cap} holds, then, for all $\sigma \in \Sigma$ and $x \in \Xsf$,
     the stopping time $\tau^\sigma$ is finite $\PP_x$-almost surely, and
     $g_\sigma(x) \leq C + \|e\|$.
\end{lemma}

\begin{proof}
    Fix $\sigma \in \Sigma$ and $x \in \Xsf$.  By \eqref{eq:sbou}, we have
    $\tau^\sigma \leq \bar \tau$, so $\tau^\sigma$ is finite $\PP_x$-almost
    surely.  Since $c \geq 0$, this inequality also gives
    $\sum_{t=0}^{\tau^\sigma - 1} c(X_t) \leq \sum_{t=0}^{\bar \tau - 1}
    c(X_t)$.  Moreover, $e(X_{\tau^\sigma}) \leq \|e\|$.  Taking expectations
    and using \eqref{eq:capc} yields $g_\sigma(x) \leq C + \|e\|$.
\end{proof}

\subsubsection{Policy Operators and ADP formulation}\label{sss:ndpolop}

For each $\sigma \in \Sigma$, we define the policy operator
$T_\sigma$ on $b\Xsf_+$ via
\begin{equation}\label{eq:ndpol}
    (T_\sigma \, g)(x)  
    = \sigma(x) e(x) + (1 - \sigma(x))
        \left[c(x) + \int g(x') P(x, \diff x') \right].
\end{equation}
It is useful to write in operator notation
\begin{equation}\label{eq:tsk}
    T_\sigma \, g = \sigma e + (1-\sigma) c +  K_\sigma \, g
    \qquad (g \in b\Xsf),
\end{equation}
where
\begin{equation}\label{eq:ksig}
    (K_\sigma \, g)(x) \coloneq (1 - \sigma(x)) \int g(x') P(x, \diff x')
    \qquad (x \in \Xsf).
\end{equation}
In \eqref{eq:tsk}, expressions such as $\sigma e$ are understood as pointwise
products.

The following lemma identifies the fixed point of each policy operator.

\begin{lemma}\label{l:tsfg}
    If \cref{a:cap} holds, then, for every $\sigma \in \Sigma$, the
    expected cost function $g_\sigma$ induced by $\sigma$ is 
    a fixed point of $T_\sigma$ in $b\Xsf_+$.
\end{lemma}

\begin{proof}
    Fix $\sigma \in \Sigma$. By \cref{l:fomc}, $\tau^\sigma$ is finite with
    probability one, so $g_\sigma$ is well-defined, and $g_\sigma \in
    b\Xsf_+$ by the bound in \cref{l:fomc} and the measurability remark after
    \eqref{eq:lcf}. If $\sigma(x) = 1$, then
    $\tau^\sigma = 0$ $\PP_x$-almost surely and
        $g_\sigma(x) = e(x) = (T_\sigma g_\sigma)(x)$.
    If $\sigma(x) = 0$, then $\tau^\sigma \geq 1$ $\PP_x$-almost surely, 
    and the Markov property and the definition of $g_\sigma$ give
    \begin{equation*}
        g_\sigma(x)
         = c(x) + \EE_x\left[
        \sum_{t=1}^{\tau^\sigma - 1} c(X_t) + e(X_{\tau^\sigma})
        \right]
         = c(x) + \EE_x[g_\sigma(X_1)]
         = (T_\sigma g_\sigma)(x).
    \end{equation*}
    Hence $T_\sigma g_\sigma = g_\sigma$.
\end{proof}

Let $\TT$ be the set of all policy operators, as defined in \eqref{eq:ndpol},
indexed over the restricted policy set $\Sigma$ defined in \eqref{eq:ndsig}.
Since each $T_\sigma$ is order preserving, the pair $(b\Xsf_+, \TT)$ forms an
ADP.  Here $b\Xsf_+$ is paired with the pointwise partial order $\leq$ and
the topology of pointwise convergence, under which a net $(f_\alpha)$
converges to $f$ if and only if $f_\alpha(x) \to f(x)$ for every $x \in
\Xsf$.  (This topology is weaker than the one induced by the supremum norm,
and the weakening is essential; see Section~\ref{ss:noncon}.)  If $f_\alpha
\leq g_\alpha$ for all $\alpha$, $f_\alpha \to f$ and $g_\alpha \to g$, then
$f(x) \leq g(x)$ for every $x$, so the pointwise order is closed and
$b\Xsf_+$ is a pospace.  Moreover, $b\Xsf_+$ is
monotone sigma-complete, since the pointwise infimum of a decreasing
sequence in $b\Xsf_+$ is again an element of $b\Xsf_+$, and is the infimum of
the sequence in $b\Xsf_+$.  For this ADP and given $g \in b\Xsf_+$, 
a policy $\sigma \in \Sigma$ is $g$-greedy when $T_\sigma \, g \leq T_s \, g$ for
all $s \in \Sigma$.  It is easy to check that one such policy always exists.
Indeed, such a policy can be found by setting
\begin{equation*}
    \sigma(x) =
    \1
    \left\{ 
        e(x) \leq  c(x) + \int g(x') P(x, \diff x') 
    \right\}
    \qquad (x \in \Xsf).
\end{equation*}
This policy is in $\Sigma$, since $\sigma$ is $\bB$-measurable and, in addition,
since $g \geq 0$,
\begin{equation*}
    e(x) \leq  c(x) 
    \implies
    e(x) \leq  c(x) + \int g(x') P(x, \diff x'),
\end{equation*}
so that $\bar \sigma \leq \sigma$.  Moreover, it is clear that
\begin{equation*}
    (T_\sigma \, g)(x) = 
    \min
    \left\{
        e(x), c(x) + \int g(x') P(x, \diff x')
    \right\}
    \leq (T_s \, g)(x)
\end{equation*}
for all $s \in \Sigma$.

This argument also proves that the ADP $(b\Xsf_+, \TT)$ is regular.

In essence, a $g$-greedy policy treats $g$ as a loss function, using it to
assign a total expected cost to each state, and makes the current best
choice accordingly.

We define the \navy{minimum loss function} $\gmin$ via
\begin{equation}\label{eq:gstar}
    \gmin(x) \coloneq \inf_{\sigma \in \Sigma} g_\sigma(x)
    \qquad (x \in \Xsf).
\end{equation}
The function $\gmin$ also takes values in $[0,\infty)$ and is
well-defined everywhere on $\Xsf$.  Below we show that $\gmin$ is equal to the value function
$v^*$ of the ADP, as defined in Section~\ref{ss:opt}.

A policy $\sigma \in \Sigma$ is called \navy{optimal} if $g_\sigma \leq g_s$ for
all $s \in \Sigma$. This is equivalent to the statement that $\sigma$ attains
the minimum possible cost from every state, and is equivalent to the ADP
definition in Section~\ref{ss:opt}. 

\subsubsection{Optimality}

We can now apply the abstract results from Section~\ref{s:or}.
The main step in the proof of \cref{t:ndbk} is to show that the no-discount optimal stopping ADP $(b\Xsf_+, \TT)$ is globally stable.

\begin{theorem}\label{t:ndbk}
    If \cref{a:cap} holds, then
    \begin{enumerate}
        \item the fundamental optimality properties hold and
        \item the VFI, OPI, and HPI algorithms all converge.
    \end{enumerate}
\end{theorem}

\begin{proof}
    We first show that $(b\Xsf_+, \TT)$ is globally stable under the topology
    of pointwise convergence.  To this end, fix $\sigma \in \Sigma$.  By
    \cref{l:tsfg}, $g_\sigma$ is a fixed point of
    $T_\sigma$ in $b\Xsf_+$.  Fix $f \in b\Xsf_+$.  Since $T_\sigma$ is affine
    with linear part $K_\sigma$ (see \eqref{eq:tsk}), we have $T_\sigma^n f -
    g_\sigma = T_\sigma^n f - T_\sigma^n g_\sigma = K_\sigma^n (f - g_\sigma)$
    for all $n \in \NN$.  To bound the right-hand side, fix $h \in b\Xsf$
    and $x_0 \in \Xsf$.  Iterating with $K_\sigma$ and applying the triangle
    inequality, we obtain
    \begin{multline*}
        |(K^n_\sigma h)(x_0)|
        \leq (1-\sigma(x_0))
        \int (1-\sigma(x_1))
        \cdots
        \int (1-\sigma(x_{n-1}))
        \\
        \int 
        |h(x_n)| P(x_{n-1}, \diff x_n)
        P(x_{n-2}, \diff x_{n-1})
        \cdots
        P(x_0, \diff x_1).
    \end{multline*}
    Hence, if $(X_t)$ is $P$-Markov and starts at $x_0$, then, with
    $\tau^\sigma$ as defined in \eqref{eq:tau},
    \begin{align*}
        |(K^n_\sigma h)(x_0)|
        & \leq \| h\|
        \int \cdots \int
        \prod_{t=0}^{n-1} (1 - \sigma(x_t))
        P(x_{n-1}, \diff x_n)
        P(x_{n-2}, \diff x_{n-1})
        \cdots
        P(x_0, \diff x_1)
        \\
        & = \| h\| \cdot \PP_{x_0} \bigcap_{t=0}^{n-1} \, \{\sigma(X_t) = 0\}
          = \| h\| \cdot \PP_{x_0} \{\tau^\sigma \geq n\}.
    \end{align*}
    Applying this bound with $h = f - g_\sigma$ gives
    \begin{equation*}
        | (T_\sigma^n f)(x) - g_\sigma(x) |
        \leq \| f - g_\sigma \| \cdot \PP_x \{ \tau^\sigma \geq n \}
        \qquad (x \in \Xsf, \; n \in \NN).
    \end{equation*}
    By \cref{l:fomc}, $\tau^\sigma$ is finite $\PP_x$-almost surely, so
    $\PP_x \{ \tau^\sigma \geq n \} \to 0$ as $n \to \infty$ for each $x$.
    Hence $T_\sigma^n f \to g_\sigma$ pointwise.  Finally, if $f \in
    b\Xsf_+$ is any fixed point of $T_\sigma$, then $f = T_\sigma^n f \to
    g_\sigma$, so $f = g_\sigma$.  Thus $g_\sigma$ is the unique fixed point
    of $T_\sigma$ in $b\Xsf_+$, and $T_\sigma$ is globally stable.

    We saw in Section~\ref{sss:ndpolop} that $(b\Xsf_+, \TT)$ is regular, and
    we have just shown that it is globally stable.  The ADP is order bounded, as well, since $T_\sigma 0 \geq 0$
    for all $\sigma$.  Since $b\Xsf_+$ is a monotone sigma-complete
    pospace (Section~\ref{sss:ndpolop}), all of the conditions of \cref{t:tspo} are satisfied.  This
    implies the conclusions of \cref{t:ndbk}.
\end{proof}

With \cref{t:ndbk} in hand, we can confirm the claim made after
\eqref{eq:gstar} that $\gmin = v^*$.  (The only potential difficulty is
that the pointwise infimum in \eqref{eq:gstar} is over an uncountable family,
so measurability of $\gmin$ is not automatic.)  By the fundamental
optimality properties, there exists an optimal policy $\sigma$, and,
for this policy, $v_\sigma = v^*$.  Since $v_\tau = g_\tau$ for all $\tau
\in \Sigma$ (\cref{l:tsfg} and the proof of \cref{t:ndbk}), we have $v^* = g_\sigma
\geq \gmin$ pointwise.  At the same time, $v^*$ is a lower bound of
$\setntn{g_\tau}{\tau \in \Sigma}$, so $v^* \leq \gmin$ pointwise.  Hence
$\gmin = v^*$; in particular, $\gmin \in b\Xsf_+$.

\Cref{a:cap} does not imply that the no-discount optimal stopping problem is
contractive.  In Section~\ref{ss:noncon} we give an example in which
\cref{a:cap} holds but neither the policy operator nor the Bellman operator
are eventually contracting with respect to any norm equivalent to the supremum
norm.  In such cases \cref{t:ndbk} still applies, while contraction-based
methods are not available.

\subsection{Sequential Analysis}\label{ss:sar}

In this section we study the classical Bayesian sequential analysis problem
introduced by \cite{wald1947sequential} and \cite{arrow1949bayes} (see the
introduction for background and related literature). The problem
involves no discounting, so the Bellman operator is not a contraction with
respect to the supremum norm.  Using \cref{t:ndbk}, our
optimality result for no-discount optimal stopping, we show that the
fundamental optimality properties hold and that VFI, OPI, and
HPI all converge under the weak assumption that the candidate densities
$f_0$ and $f_1$ are distinct.

We proceed by first reducing the sequential analysis problem to a special case
of the no-discount optimal stopping problem treated in \cref{t:ndbk}, then
verifying the conditions of that theorem.

\subsubsection{Setup}

A decision-maker observes a sequence $Z_1, Z_2, \ldots$ of {\sc iid} samples
drawn from an unknown probability density $f$ on $\RR$.  It is known that
$f$ is either $f_0$ or $f_1$, where $f_0$ and $f_1$ are given probability
densities.  After each observation the decision-maker chooses one of three
actions: accept the hypothesis $f = f_0$ and stop, accept the hypothesis
$f = f_1$ and stop, or draw an additional sample at cost $c > 0$.  Accepting
the wrong hypothesis incurs a strictly positive loss: $L_0$ when $f_0$ is
accepted but $f = f_1$, and $L_1$ when $f_1$ is accepted but $f = f_0$.  The
decision-maker holds a prior $\pi_0 \in (0, 1)$ that $f = f_1$, and updates to
the posterior $\pi_n$ after $n$ observations via Bayes' rule.  This posterior
is the natural state variable for the dynamic program below.

The state space for the belief state $\pi$ is
$\Xsf = [0,1]$ and the action space is $\Asf = \{0, 1, 2\}$, where action $0$
represents accepting $f_0$, action $1$ represents accepting $f_1$, and action
$2$ represents continuing to sample. We assume that the two densities are
defined on $\RR$.

The Bellman equation has the form
\begin{equation}\label{eq:waldie2}
    g(\pi) = \min
    \left\{
        \pi L_0, \; (1-\pi) L_1, \; c + \int g(\pi') P(\pi, \diff \pi')
    \right\}
\end{equation}
for $\pi \in [0,1]$, where the stochastic kernel $P$ obeys
\begin{equation}\label{eq:waldsk}
    (Pg)(\pi) 
    \coloneq \int g(\kappa(\pi, z)) \psi(\pi, z) \diff z.
\end{equation}
Here
\begin{equation}\label{eq:predden}
    \psi(\pi,z) = (1-\pi)f_0(z) + \pi f_1(z)
\end{equation}
is the predictive density and
\begin{equation*}
    \kappa(\pi, z)
    = \frac{\pi f_1(z)}{(1-\pi) f_0(z) + \pi f_1(z)}
    = \frac{\pi f_1(z)}{\psi(\pi, z)}
\end{equation*}
is the Bayesian update rule.  (On the set where $\psi(\pi, z) = 0$, which
has probability zero under $\psi(\pi, \cdot)$, we set $\kappa(\pi, z)
\coloneq \pi$; this choice has no effect on $P$.)  Note that $\kappa$ takes
values in $[0,1]$, and that the endpoints can be reached even when $\pi \in
(0,1)$: an observation $z$ with $f_1(z) = 0 < f_0(z)$ yields $\kappa(\pi, z)
= 0$, while one with $f_0(z) = 0 < f_1(z)$ yields $\kappa(\pi, z) = 1$.
This is why the endpoints are included in $\Xsf$.  Together, $\psi$ and $\kappa$ define the
stochastic kernel $P$ governing the belief state $(\pi_n)_{n \geq 0}$,
describing how beliefs evolve from the perspective of the controller when the
observation sequence $(Z_n)_{n \geq 1}$ is forecast using the predictive
density.

In all of what follows, the cost $c$ and the losses $L_0$ and $L_1$ are assumed
to be positive constants. Our aim is to characterize and solve for optimal policies.

In terms of dynamic programming, we can simplify the sequential analysis problem to a
binary stopping problem.  Our first step is to set
\begin{equation}\label{eq:epi}
    e(\pi) \coloneq \min\{\pi L_0, \; (1-\pi) L_1\}.
\end{equation}
Now consider the Bellman equation
\begin{equation}\label{eq:waldie3}
    g(\pi) = \min
    \left\{
        e(\pi), \; c + (Pg)(\pi)
    \right\}.
\end{equation}
We claim that solving this dynamic program is sufficient for solving the sequential
analysis problem. To see this, suppose we are able to show that the fundamental
optimality properties hold for this dynamic program.  Let the value
function be denoted by $\gmin$. Continuing the convention that the controller
always stops when indifferent between stopping and continuing, Bellman's
principle of optimality tells the controller to stop if and only if $e(\pi)
\leq c + (P\gmin)(\pi)$. 

In the present setting, this means that the controller stops if and only if at
least one of the stopping losses $\pi L_0$ and $(1-\pi)L_1$ is less than or
equal to the continuation loss $c + (P\gmin)(\pi)$. When this stop occurs, the
controller then makes the static choice over the two density options (selecting
$f_0$ or $f_1$) depending on which of $\pi L_0$ and $(1-\pi)L_1$ is smaller.
Since this static problem is trivial, we can concentrate on solving the dynamic
problem represented by the Bellman equation \eqref{eq:waldie3}.

The stopping problem just described is a special case of the no-discount optimal
stopping problem from Section~\ref{ss:nodos}, with $e$ as the exit cost function in
\eqref{eq:epi} and the flow cost $c$ constant over the state space $\Xsf =
[0,1]$.  (To confirm this, compare the Bellman equation \eqref{eq:waldie3} with
the general case in \eqref{eq:bell}.) As a result, \cref{t:ndbk} applies.
All we need to do is check that \cref{a:cap} is valid in the current
setting.  This turns out to be true whenever $f_0$ and $f_1$ are distinct.
Our proof will rely on a bound for martingale stopping times in \cref{t:exit}.

\subsubsection{Optimality for the Stopping Problem}

Let $\bar
\tau$ be the upper bound stopping time for the belief state (i.e., the stopping
time in \eqref{eq:bartau} specialized to the current setting).
This stopping time is defined in terms of the certain exit
region (see \eqref{eq:cerex}), which, for our problem, is
\begin{equation*}
    \bar E \coloneq
    \setntn{\pi \in [0,1]}{\min\{\pi L_0, \; (1-\pi) L_1\} \leq c}.
\end{equation*}
Equivalently, 
\begin{equation*}
    \pi \in \bar E 
    \iff
    \pi \leq \frac{c}{L_0} \quad \text{or} \quad \pi \geq 1 - \frac{c}{L_1}.
\end{equation*}
The lower bound policy $\bar \sigma$ is (recalling its definition from
\cref{ss:nodos}) the indicator function for $\bar E$, stopping the process
whenever the belief state enters the certain exit region.  The upper bound
stopping time is, therefore,
\begin{equation}\label{eq:btb}
    \bar \tau = \inf
    \left\{
        n \geq 0 \,:\, \pi_n \leq \frac{c}{L_0} \text{ or } \pi_n \geq 1 -
        \frac{c}{L_1}
    \right\}.
\end{equation}
Here the $(\pi_n)$ process evolves according to the kernel $P$ from
\eqref{eq:waldsk}: given $\pi_n$, we draw $Z_{n+1}$ independently from
$\psi(\pi_n, \cdot)$, and then set
\begin{equation}\label{eq:bsu}
    \pi_{n+1}
        = \kappa(\pi_n, Z_{n+1})
        = \frac{\pi_n f_1(Z_{n+1})}{\psi(\pi_n, Z_{n+1})}.
\end{equation}
(Note here that division by zero is not a concern: since $Z_{n+1}$ is drawn
from $\psi(\pi_n, \cdot)$, we have $\psi(\pi_n, Z_{n+1}) > 0$ almost surely.)

Matching \eqref{eq:ndsig}, the policy set $\Sigma$ under consideration is all
$\bB$-measurable $\sigma \colon \Xsf \to \{0,1\}$ with $\bar \sigma \leq
\sigma$.  For $\sigma$ in this set, the policy operator takes the form
\begin{equation}\label{eq:ndpoll}
    (T_\sigma \, g)(\pi)  
    = \sigma(\pi) e(\pi) + (1 - \sigma(\pi))
        \left[c + \int g(\pi') P(\pi, \diff \pi') \right].
\end{equation}
Let $V$ be all bounded Borel measurable functions from $[0,1]$ to $\RR_+$.
With $\TT_{\rm SA}$ as the family of operators described by \eqref{eq:ndpoll},
indexed by $\sigma \in \Sigma$, the pair $(V, \TT_{\rm SA})$ forms an ADP.
For this ADP, we can state the following result. In the statement, the
densities $f_0$ and $f_1$ are called \navy{distinct} if they differ on a set
of positive Lebesgue measure.

\begin{proposition}\label{p:waldop2}
    If $f_0$ and $f_1$ are distinct, then the fundamental optimality properties hold and VFI,
    OPI, and HPI all converge.
\end{proposition}

To prove \cref{p:waldop2}, we recognize $(V, \TT_{\rm SA})$ as a special case of the
no-discount optimal stopping ADP treated in \cref{t:ndbk}.  As such, we only
need to verify \cref{a:cap}.  Since the flow cost $c$ is constant, and hence
bounded, the sufficient condition \eqref{eq:unifexit} applies, so it suffices
to show that $\sup_{0 \leq \pi \leq 1} \EE_\pi \bar \tau$ is finite.

As a first step, we recall that, given two probability densities $f_0$ and $f_1$
on $\RR$, the \navy{triangular discrimination} is
\begin{equation*}
    \Delta(f_0, f_1) \coloneq \int \frac{[f_1(z) - f_0(z)]^2}{f_0(z) + f_1(z)} \, \diff z,
\end{equation*}
where the integrand is defined to be zero when $f_0(z) + f_1(z) = 0$.

\begin{lemma}\label{l:delta_positive}
    If $f_0$ and $f_1$ are distinct, then $\Delta(f_0, f_1) > 0$.
\end{lemma}

\begin{proof}
    If $f_0$ and $f_1$ are distinct, then they differ on a set $E$ of positive
    measure.  For
    $z \in E$ we have $f_0(z) + f_1(z) > 0$ (since otherwise both are zero,
    which contradicts the definition of $E$).  This means that the integrand in
    the definition of $\Delta(f_0, f_1)$ is positive on $E$.  Integration of a positive
    function over a set of positive measure yields a positive value.
\end{proof}

Now let's investigate the properties of stopping times for the process $(\pi_n)$
from \eqref{eq:bsu}. We will begin with a generic stopping time 
\begin{equation*}
    \tau = \inf\{n \geq 0 : \pi_n \notin (a,b)\},
\end{equation*}
where $a, b$ are numbers satisfying $0 < a < b < 1$.

\begin{lemma}\label{l:belief}
    If $f_0$ and $f_1$ are distinct, then, for any initial condition $\pi_0 \in [0,1]$, we have
    \begin{equation*}
        \EE_{\pi_0}[\tau] \leq \frac{1}{\delta},
        \quad \text{where} \quad
        \delta \coloneq a^2 (1-b)^2 \Delta(f_0, f_1).
    \end{equation*}
\end{lemma}

\begin{proof}
    Fix $\pi_0 \in [0,1]$. If $\pi_0$ is not in $(a,b)$, then $\tau = 0$, in
    which case the claim is trivial. Hence, from now on, we assume that $\pi_0$
    is in $(a,b)$.  Note that $(\pi_n)_{n \geq 0}$ is a martingale, since
    \begin{equation*}
        \EE[\pi_{n+1} \mid \pi_n = \pi]
        = \int \frac{\pi f_1(z)}{\psi(\pi, z)} \cdot \psi(\pi, z) \, \diff z
        = \int \pi f_1(z) \, \diff z = \pi.
    \end{equation*}
    We apply \cref{t:exit} to this bounded martingale. Our main task is to
    obtain $\delta$ in \eqref{eq:var_bound}. For fixed $\pi$, the law of motion
    \eqref{eq:bsu} yields
    \begin{equation*}
        \EE[\pi^2_{n+1} \mid \pi_n = \pi]
        = \int \frac{\pi^2 f_1(z)^2}{\psi(\pi, z)} \, \diff z
        = \pi^2 \int \frac{f_1(z)^2}{\psi(\pi, z)} \, \diff z.
    \end{equation*}
    (As for the triangular discrimination, the integrand is defined to be
    zero when $\psi(\pi, z) = 0$, which occurs only when $f_0(z) = f_1(z) = 0$).
    Using the fact that $(\pi_n)$ is a martingale, we get 
    \begin{equation*}
        v(\pi) \coloneq \EE[(\pi_{n+1} - \pi_n)^2 \mid \pi_n = \pi]
        = \pi^2 \left[\int \frac{f_1(z)^2}{\psi(\pi, z)} \, \diff z - 1\right].
    \end{equation*}
    To simplify the bracketed term, we observe that
    \begin{align*}
        \int \frac{[f_1(z) - \psi(\pi, z)]^2}{\psi(\pi, z)} \, \diff z
        &= \int \frac{f_1(z)^2}{\psi(\pi, z)} \, \diff z
            - 2\int f_1(z) \, \diff z + \int \psi(\pi, z) \, \diff z \\
        &= \int \frac{f_1(z)^2}{\psi(\pi, z)} \, \diff z - 1.
    \end{align*}
    Using this fact, combined with $f_1(z) - \psi(\pi, z) = (1-\pi)[f_1(z) -
    f_0(z)]$ and the definition of $v(\pi)$, gives
    \begin{equation}\label{eq:var_formula}
        v(\pi) = \pi^2 (1-\pi)^2 \int \frac{[f_1(z) - f_0(z)]^2}{\psi(\pi, z)} \, \diff z.
    \end{equation}
    Since $\psi(\pi, z) \leq f_0(z) + f_1(z)$, we have
    \begin{equation}\label{eq:lower_bound}
        \int \frac{[f_1(z) - f_0(z)]^2}{\psi(\pi, z)} \, \diff z
        \geq \int \frac{[f_1(z) - f_0(z)]^2}{f_0(z) + f_1(z)} \, \diff z
        = \Delta(f_0, f_1).
    \end{equation}
    Combining \eqref{eq:var_formula} and \eqref{eq:lower_bound}, for $\pi \in [a,b]$:
    \begin{equation*}
        v(\pi) \geq \pi^2 (1-\pi)^2 \Delta(f_0, f_1) \geq a^2 (1-b)^2 \Delta(f_0, f_1) = \delta.
    \end{equation*}
    By \cref{l:delta_positive}, the triangular discrimination is strictly
    positive, so $\delta > 0$. Since $v(\pi) = \EE[(\pi_{n+1} - \pi_n)^2 \mid
    \pi_n = \pi]$, this verifies condition~\eqref{eq:var_bound} for $M_n =
    \pi_n$ and this choice of $\delta$. Applying \cref{t:exit} and using
    $(\pi_\tau - \pi_0)^2 \leq 1$, we get
    \begin{equation*}
        \EE_{\pi_0}[\tau] \leq \frac{\EE_{\pi_0}[(\pi_\tau - \pi_0)^2]}{\delta}
        \leq \frac{1}{\delta}.
    \end{equation*}
    This verifies the claim in \cref{l:belief}.
\end{proof}

\begin{proof}[Proof of \cref{p:waldop2}]
    Let $f_0$ and $f_1$ be distinct. Since $(V, \TT_{\rm SA})$ is a special case
    of the no-discount optimal stopping ADP treated in \cref{t:ndbk}, it
    suffices, by \eqref{eq:unifexit}, to show that $\sup_{\pi \in [0,1]}
    \EE_\pi \bar \tau$ is finite, where $\bar \tau$ is as defined in
    \eqref{eq:btb}.

    Consider first the degenerate case $c/L_0 + c/L_1 \geq 1$.  In this case
    the two intervals $[0, c/L_0]$ and $[1 - c/L_1, 1]$ cover $[0,1]$, so
    $\bar E = [0,1]$ and $\bar\tau = 0$ $\PP_\pi$-a.s.\ for every $\pi \in [0,1]$.
    Hence $\sup_{\pi \in [0,1]} \EE_\pi \bar\tau = 0 < \infty$.

    Now consider the non-degenerate case $c/L_0 + c/L_1 < 1$.  Setting
    $a \coloneq c/L_0$ and $b \coloneq 1 - c/L_1$, we have $0 < a < b < 1$.
    By construction, the stopping time $\bar\tau$ in \eqref{eq:btb} is the
    first exit time of the belief process from $(a,b)$, so \cref{l:belief}
    applies and yields $\EE_{\pi}[\bar\tau] \leq 1/\delta$, where
    \begin{equation*}
        \delta
            \coloneq a^2 (1-b)^2 \Delta(f_0, f_1)
            = \left(\frac{c}{L_0}\right)^2 \left(\frac{c}{L_1}\right)^2 \Delta(f_0, f_1)
            = \left(\frac{c^2}{L_0L_1}\right)^2 \Delta(f_0, f_1).
    \end{equation*}
    Since $f_0$ and $f_1$ are distinct, \cref{l:delta_positive} implies that
    $\Delta(f_0, f_1) > 0$, so $\delta > 0$. The bound
    $\EE_{\pi}[\bar\tau] \leq 1/\delta$ is valid for all $\pi \in [0, 1]$,
    confirming that \cref{a:cap} holds.
    As a result, all the claims in \cref{t:ndbk} are valid.
\end{proof}

\section{Conclusion}\label{s:conclusion}

We showed that the fundamental optimality properties of dynamic programming,
together with convergence of value function iteration, optimistic policy
iteration and Howard policy iteration, follow from global stability of the
policy operators on a pospace, without contractivity.  For optimal stopping
without discounting, this permits unbounded expected termination times, where
contraction methods are not available. Several open problems remain.  One is to
handle unbounded costs, as in optimal stopping with unbounded exit costs or
stochastic shortest path problems with unbounded one-period costs, for example
by working in weighted function spaces without reintroducing the uniform bounds
that our approach is designed to avoid.  Another is to relax the assumption that
greedy policies exist in \cref{t:pospace,t:tspo}, which can fail when action
sets are infinite, and to determine which of our conclusions survive when greedy
policies exist only on part of the value space.

\section{Appendix}

This appendix gives an example illustrating the distinction
between global stability and order continuity, establishes
that \cref{a:cap} does not imply contractivity, 
and proves a bound on exit times for bounded martingales,
which is used in the sequential analysis application of
Section~\ref{ss:sar}.

\subsection{Global stability without order continuity}

The following example shows that global stability,
the key stability condition in \cref{t:tspo}, does not
imply order continuity used in Theorem~5.5 of \cite{sargent2025partially}.

\begin{example}\label{eg:gswoc}
    Let $V=[0,1]$ with the usual order and topology, and
    consider an ADP with a singleton policy set
    $\Sigma=\{\sigma\}$. Define its policy operator
    $T_\sigma:V\to V$ by
    \begin{equation*}
        (T_\sigma v)
        =
        \begin{cases}
            1/2, & v=0,\\
            1, & v>0.
        \end{cases}
    \end{equation*}
    The operator $T_\sigma$ is order preserving and satisfies
    \begin{equation*}
        T_\sigma^2 v=1
        \qquad\text{for every }v \in V.
    \end{equation*}
    Hence it is globally stable, with unique fixed point
    $v_\sigma=1$.
    
    However, $T_\sigma$ is not order continuous. Indeed,
    taking $v_n = 1/n$, we have $v_n \downarrow 0$, whereas
    \begin{equation*}
        T_\sigma v_n = 1 \longrightarrow 1
        \neq \frac12 = T_\sigma 0.
    \end{equation*}
    Nevertheless, all the assumptions of \cref{t:tspo} hold.
    The space $V$ is a monotone sigma-complete pospace,
    and regularity is automatic since there is only one
    policy. Moreover, order boundedness holds with $b=0$,
    since
    \begin{equation*}
        0\leq T_\sigma 0=\frac12.
    \end{equation*}
    Thus, \cref{t:tspo} applies even though the policy operator
    is not order continuous.
\end{example}

\subsection{Failure of Contractivity}\label{ss:noncon}

In this section we show that \cref{a:cap} does not imply contractivity of the
no-discount optimal stopping problem of Section~\ref{ss:nodos}.  We give an
example in which \cref{a:cap} holds, so that \cref{t:ndbk} applies, but
neither the policy operator $T_{\bar \sigma}$ nor the Bellman operator $T$ is eventually
contracting with respect to any norm equivalent to the supremum norm, and the
policy operator $T_{\bar \sigma}$ is not globally stable in the topology induced by the
supremum norm.  Here we call a self-map $S$ on $b\Xsf_+$ \navy{eventually
contracting} with respect to a norm $\| \cdot \|'$ on $b\Xsf$ if there exist
$N \in \NN$ and $\lambda < 1$ with $\|S^N f - S^N g\|' \leq \lambda \|f -
g\|'$ for all $f, g \in b\Xsf_+$.

\begin{example}\label{eg:noncon}
    Let $\Xsf = \{0, 1, 2, \ldots\}$, let $\bB$ be the set of all subsets of
    $\Xsf$, and let $P(k, \cdot) = \delta_{k-1}$ for $k \geq 1$ and
    $P(0, \cdot) = \delta_0$, so that the state moves deterministically
    towards $0$.  Let $e(0) = 0$, $e(k) = 1$ for $k \geq 1$, and $c(k) =
    2^{-k}$ for all $k$.  Then $\bar E = \{0\}$, since $e(0) = 0 \leq 1 =
    c(0)$ and $e(k) = 1 > 2^{-k} = c(k)$ for $k \geq 1$.  Starting from $k$,
    the state reaches $0$ at time $k$, so $\bar \tau = k$ $\PP_k$-almost
    surely and
    \begin{equation*}
        \EE_k \sum_{t=0}^{\bar \tau - 1} c(X_t)
        = \sum_{j=1}^k 2^{-j} < 1.
    \end{equation*}
    Hence \cref{a:cap} holds, and \cref{t:ndbk} applies, even though
    $\sup_k \EE_k \bar \tau = \infty$.

    We claim that neither $T_{\bar \sigma}$ nor $T$ is eventually
    contracting with respect to any norm on $b\Xsf$ equivalent to $\| \cdot
    \|$.  To see this, fix $\epsilon \in (0, 1/2)$, let $f \equiv 0$ and $g
    \equiv \epsilon$, and let $S$ be either $T_{\bar \sigma}$ or $T$.  For
    $k \geq 1$ and $h \in b\Xsf_+$, we have $(T_{\bar \sigma} h)(k) = 2^{-k} +
    h(k-1)$ and $(T h)(k) = \min\{1, 2^{-k} + h(k-1)\}$.  Let $s_{n,k}
    \coloneq \sum_{j = k-n+1}^k 2^{-j}$, and note that $s_{n,k} \leq
    2^{-(k-n)} \leq 1/2$ whenever $k \geq n+1$.  Induction on $n$ shows that,
    for $h \in \{f, g\}$, $n \in \NN$ and $k \geq n + 1$,
    \begin{equation*}
        (S^n h)(k) = s_{n,k} + h(k - n),
    \end{equation*}
    since, in the case $S = T$, the right-hand side is at most $1/2 +
    \epsilon < 1$, so the minimum in the definition of $T$ is not active.
    Hence $(S^n g)(k) - (S^n f)(k) = \epsilon$ for all $k \geq n+1$, and so
    $\|S^n g - S^n f\| \geq \epsilon = \|g - f\|$ for all $n \in \NN$.

    Now suppose that $\| \cdot \|'$ is a norm with $a \|h\| \leq \|h\|' \leq
    b \|h\|$ for all $h \in b\Xsf$, where $0 < a \leq b$, and that $S$ is
    eventually contracting with respect to $\| \cdot \|'$, with constants $N$
    and $\lambda$.  Then, for all $j \in \NN$,
    \begin{equation*}
        \epsilon
        \leq \|S^{jN} g - S^{jN} f\|
        \leq \frac{1}{a} \|S^{jN} g - S^{jN} f\|'
        \leq \frac{\lambda^j}{a} \|g - f\|'
        \leq \frac{b}{a} \lambda^j \epsilon,
    \end{equation*}
    which is impossible for large $j$.  This proves the claim.  The same
    computation shows that $T_{\bar \sigma}$ is not globally stable under the
    topology induced by $\| \cdot \|$: since $\|T^n_{\bar \sigma} g -
    T^n_{\bar \sigma} f\| \geq \epsilon$ for all $n$, the sequences
    $(T^n_{\bar \sigma} f)$ and $(T^n_{\bar \sigma} g)$ cannot both converge to
    $g_{\bar \sigma}$ in the supremum norm.
\end{example}

\subsection{A Stopping Time Bound}\label{ss:stb}

Here we state a result on exit times for bounded martingales that is used in the
proof of \cref{p:waldop2}. In the statement of the theorem, $(M_t, \fF_t)$ is a
discrete-time martingale with $M_0 = m$, and
\begin{equation*}
    \tau = \inf\{t \geq 0 : M_t \notin (a,b)\}
\end{equation*}
for some $a,b$ in $\RR$ with $a < b$.

\begin{theorem}\label{t:exit}
    Let $(M_t, \fF_t)$ be bounded, in the sense that there exists a $K \in \RR$
    with $|M_t| \leq K$ with probability one for all $t$.  If, in addition,
    there exists a $\delta > 0$ with
    \begin{equation}\label{eq:var_bound}
        \delta \leq \EE[(M_{t+1} - M_t)^2 \mid \fF_t] \quad \text{for all } t < \tau,
    \end{equation}
    then
    \begin{equation*}
        \EE_m[\tau] \leq \frac{\EE_m[(M_\tau - m)^2]}{\delta}.
    \end{equation*}
\end{theorem}

\begin{proof}
    Throughout the proof, all expectations are conditional on $M_0 = m$, and
    we write $\EE$ for $\EE_m$.  Define the process
    $\langle M \rangle_t \coloneq \sum_{s=0}^{t-1} (M_{s+1} - M_s)^2$,
    with $\langle M \rangle_0 \coloneq 0$.
    We claim that $(M_t - M_0)^2 - \langle M \rangle_t$ is a martingale.
    To see this we write $M_{t+1} - M_0 = (M_{t+1} - M_t) + (M_t - M_0)$ and
    square it to get
    \begin{equation*}
        (M_{t+1} - M_0)^2 = (M_t - M_0)^2 + 2(M_t - M_0)(M_{t+1} - M_t) + (M_{t+1} - M_t)^2.
    \end{equation*}
    Taking conditional expectations and using the martingale property,
    \begin{equation}\label{eq:sq_expand}
        \EE[(M_{t+1} - M_0)^2 \mid \fF_t]
            = (M_t - M_0)^2 + \EE[(M_{t+1} - M_t)^2 \mid \fF_t].
    \end{equation}
    By the definition of $\langle M \rangle$,
    $\langle M \rangle_{t+1} = \langle M \rangle_t + (M_{t+1} - M_t)^2$, so
    \begin{equation}\label{eq:qv_expand}
        \EE[\langle M \rangle_{t+1} \mid \fF_t]
            = \langle M \rangle_t + \EE[(M_{t+1} - M_t)^2 \mid \fF_t].
    \end{equation}
    Subtracting \eqref{eq:qv_expand} from \eqref{eq:sq_expand}, we get
    \begin{equation*}
        \EE[(M_{t+1} - M_0)^2 - \langle M \rangle_{t+1} \mid \fF_t]
        = (M_t - M_0)^2 - \langle M \rangle_t.
    \end{equation*}
    This is the martingale property for $(M_t - M_0)^2 - \langle M \rangle_t$.
    Since $(M_t)$ is bounded, this process is integrable for each $t$, and hence
    is a martingale. Since $\tau \wedge n \leq n$ is a bounded stopping time, the
    optional stopping theorem gives
    \begin{equation*}
        \EE[(M_{\tau \wedge n} - M_0)^2 - \langle M \rangle_{\tau \wedge n}]
        = \EE[(M_0 - M_0)^2 - \langle M \rangle_0] = 0,
    \end{equation*}
    and hence
    \begin{equation}\label{eq:ost_applied}
        \EE[(M_{\tau \wedge n} - M_0)^2] = \EE[\langle M \rangle_{\tau \wedge n}].
    \end{equation}
    In addition, for $t < \tau$, we have
    \begin{equation*}
        \EE[\langle M \rangle_{t+1} - \langle M \rangle_t \mid \fF_t]
        = \EE[(M_{t+1} - M_t)^2 \mid \fF_t] \geq \delta,
    \end{equation*}
    so
    \begin{equation}\label{eq:qv_lower}
        \EE[\langle M \rangle_{\tau \wedge n}]
        = \EE\left[\sum_{t=0}^{\tau \wedge n-1} (M_{t+1} - M_t)^2\right]
        \geq \delta \cdot \EE[\tau \wedge n].
    \end{equation}
    Combining \eqref{eq:qv_lower} and \eqref{eq:ost_applied} yields
    \begin{equation*}
        \EE[\tau \wedge n]
        \leq \frac{\EE[(M_{\tau \wedge n} - M_0)^2]}{\delta}
        \leq \frac{4K^2}{\delta}
        \quad \text{for all } n \in \NN.
    \end{equation*}
    By monotone convergence, $\EE[\tau] \leq 4K^2/\delta < \infty$, so
    $\tau < \infty$ with probability one.  As a result, $M_{\tau \wedge n}
    \to M_\tau$ almost surely as $n \to \infty$.  Taking $n \to \infty$ in the
    first inequality of the last display, applying monotone convergence on the
    left and dominated convergence on the right (using boundedness of the
    martingale), gives the bound claimed in \cref{t:exit}.
\end{proof}

\bibliographystyle{plainnat}
\bibliography{qe_bib}

\end{document}